\documentclass{article}

\usepackage{amsmath,amssymb,amsthm}

\newtheorem{theorem}{Theorem}[section]
\newtheorem{definition}[theorem]{Definition}
\newtheorem{proposition}[theorem]{Proposition}
\newtheorem{corollary}[theorem]{Corollary}
\newtheorem{lemma}[theorem]{Lemma}
\newtheorem{fact}[theorem]{Remark}
\newtheorem{exemplu}[theorem]{Example}

\newcommand{\bdfn}{\begin{definition}}
\newcommand{\edfn}{\end{definition}}
\newcommand{\bthm}{\begin{theorem}}
\newcommand{\ethm}{\end{theorem}}
\newcommand{\bprop}{\begin{proposition}}
\newcommand{\eprop}{\end{proposition}}
\newcommand{\bcor}{\begin{corollary}}
\newcommand{\ecor}{\end{corollary}}
\newcommand{\blem}{\begin{lemma}}
\newcommand{\elem}{\end{lemma}}
\newcommand{\bfact}{\begin{fact}}
\newcommand{\efact}{\end{fact}}
\newcommand{\bex}{\begin{exemplu}\begin{rm}}
\newcommand{\eex}{\end{rm}\end{exemplu}}

\def\R{{\mathbb R}}
\def\N{{\mathbb N}}

\newcommand{\eps}{\varepsilon}

\newcommand{\be}{\begin{enumerate}}
\newcommand{\ee}{\end{enumerate}}
\newcommand{\bt}{\begin{tabular}}
\newcommand{\et}{\end{tabular}}
\newcommand{\beq}{\begin{equation}}
\newcommand{\eeq}{\end{equation}}
\newcommand{\ba}{\begin{array}} 
\newcommand{\ea}{\end{array}}
\newcommand {\bea} {\begin{eqnarray}}
\newcommand {\eea} {\end {eqnarray}}
\newcommand {\bua} {\begin{eqnarray*}}
\newcommand {\eua} {\end {eqnarray*}}
\newcommand{\se}{\subseteq}
\newcommand{\ds}{\displaystyle}

\newcommand{\Lra}{\Leftrightarrow}

\begin{document}

\title{Rates of asymptotic regularity for Halpern iterations of nonexpansive mappings
\footnote{The research reported in this paper was carried out during 
the author's stay 
at the Max-Planck-Institute for Mathematics (Bonn) whose support is gratefully 
acknowledged.}}
\author{{\bfseries Lauren\c tiu Leu\c stean}\\[0.2cm]
Department of Mathematics, Technische Universit\" at Darmstadt,\\
 Schlossgartenstrasse 7, 64289 Darmstadt, Germany\\[0.1cm] and\\[0.1cm]
Institute of Mathematics "Simion Stoilow'' of the 
Romanian Academy, \\
Calea Grivi\c tei 21, P.O. Box 1-462, Bucharest, Romania\\[0.1cm]
E-mail: leustean@mathematik.tu-darmstadt.de
}

\maketitle

\begin{abstract}
In this paper we obtain new effective results on the Halpern iterations of nonexpansive mappings using methods from mathematical logic or, more specifically, proof-theoretic techniques. We give effective rates of asymptotic regularity for the Halpern iterations of nonexpansive  self-mappings of nonempty convex sets in normed spaces. The paper presents another case study in the project of {\em proof mining}, which is concerned with the extraction of effective uniform bounds from (prima-facie) ineffective proofs.
\end{abstract}

\section{Introduction}

This paper presents another case study in the project of {\em proof mining}, by which we mean the logical analysis of mathematical proofs with the aim of extracting new numerically relevant information hidden in the proofs.

General logical metatheorems were obtained (using proof-theoretic methods)  in \cite{Kohlenbach-metapaper} and \cite{Gerhardy/Kohlenbach} for various classes of spaces in functional analysis and metric geometry, such as metric, hyperbolic spaces in the sense of Reich/Kirk/ Kohlenbach, CAT(0), (uniformly convex) normed  and inner product spaces. Further examples ($\R$-trees, hyperbolic spaces in the sense of Gromov and uniformly convex hyperbolic spaces) are discussed in \cite{LL-WOLLIC-06}.
These metatheorems guarantee a priorly, under very general logical conditions, the extractability  of effective bounds from large classes of proofs in functional analysis, and moreover they provide algorithms for actually extracting the bounds. The bounds are uniform for all parameters meeting very weak local boundedness conditions. We refer to Kohlenbach's forthcoming book for details \cite{K-book}.

In this paper we apply proof mining to  metric fixed point theory, more specifically to the (approximate) fixed point theory of nonexpansive mappings, one of the most active branches of nonlinear functional analysis. We refer to \cite{Kirk-handbook} for an extensive account of metric fixed point theory.

In the following, $(X,\|\cdot\|)$ is a normed space and $C$ is a nonempty convex subset of $X$. A mapping $T:C\to C$ is called {\em nonexpansive} if for all $x,y\in C$,
\[\|Tx-Ty\|\leq \|x-y\|.\]

The usual Picard iterations are not the proper iterations for nonexpansive mappings and that's why other iterations were considered in this case.  The {\em Krasnoselski-Mann iteration} \cite{Mann(53),Krasnoselski(55),Groetsch(72)} star-ting with $x\in C$ is defined by: 
\begin{equation}
x_0:=x, \quad x_{n+1}:=(1-\lambda_n)x_n+\lambda_n Tx_n \quad \text{for~} n\geq 0, \label{KM-lambda-n-def-hyp}
\end{equation}
where $(\lambda_n)_{n\geq 0}$ is a sequence in $[0,1]$.

One of the most important notions in fixed point theory is the {\em asymptotic regularity}, defined in \cite{Browder/Petryshyn(66)}, but already implicit in \cite{Krasnoselski(55),Schaefer(57),Edelstein(70)}.  A mapping $T:C\to C$ is called {\em asymptotically regular}  if  for all $x\in C$,
\[\displaystyle\lim_{n\to\infty}\|T^n(x)-T^{n+1}(x)\|=0.\]
  
\noindent For constant $\lambda_n=\lambda\in[0,1]$, the asymptotic regularity of the averaged mapping 
$T_\lambda:=(1-\lambda)I+\lambda T$ is equivalent to the fact that $\displaystyle\lim_{n\to\infty} \|x_n-Tx_n\|=0$ for all $x\in C$. Therefore, for  general $(\lambda_n)$ in $[0,1]$, a nonexpansive mapping $T$ is {\em $\lambda_n$-asymptotically regular} \cite{Borwein/Reich/Shafrir(92)} if  for all $x\in C$, 
\beq
\lim_{n\to\infty}\|x_n-Tx_n\|=0. \label{def-as-reg}
\eeq

Methods of proof mining were applied in \cite{Kohlenbach(01),Kohlenbach(03),Kohlenbach/Leustean(03),L-07-JMAA} to obtain effective rates of as-ymptotic regularity for the Krasnoselski-Mann iterations of nonexpansive mappings in normed and CAT(0)-spaces or even in the more general class of (uniformly convex) hyperbolic spaces.

In this paper, we consider other iterations, introduced in  \cite{Halpern(67)}. For $x\in C$ and $(\lambda_n)_{n\ge 1}$ in $[0,1]$, the {\em Halpern iteration} starting with $x$ is defined as:
\beq
x_0:=x, \quad x_{n+1}:=\lambda_{n+1}x+(1-\lambda_{n+1})Tx_n \,\, \text{~for~} n\geq 0. \label{Halpern-iterate-x}
\eeq

\noindent As Wittmann remarked \cite{Wittmann-Hopf,Wittmann}, if $T$ is linear and $\ds\lambda_n:=\frac1{n+1}$, then $\ds x_n=\frac{1}{n+1}\ds\sum_{i=0}^nT^ix$, so the Halpern iterations could be regarded as nonlinear generalizations of the usual Cesaro averages. 

One of the earliest and most important results on the convergence of Halpern iterations is the following one.

\begin{theorem}\label{wittmann-thm}\cite[Theorem 2]{Wittmann}\\
Let $X$ be a Hilbert space, $C\se X$  a nonempty closed convex subset, $T:C\to C$ nonexpansive and $(\lambda_n)_{n\geq 1}$ be a sequence in $\in[0,1]$ satisfying the following conditions: $\ds \lim_{n\to\infty} \lambda_n =0$, $\ds\sum_{n=1}^\infty \lambda_n$ is divergent and $\ds \sum_{n=1}^\infty|\lambda_{n+1}-\lambda_n|$ is convergent. Assume moreover that the set $Fix(T)$ of fixed points of $T$ is nonempty .

Then for any $x\in C$, the Halpern iteration $(x_n)_{n\geq 1}$ is norm convergent to the unique fixed point $Px$ of $T$ with $\|x-Px\|\leq \|x-y\|$ for any $y\in Fix(T)$ .
\end{theorem}
Generalizations of this theorem to the Banach space case and different conditions on $(\lambda_n)$ were considered in numerous papers. We refer to \cite{Xu-07-JMAA-online} for a nice exposition.

In the following, we consider the important problem of asymptotic regularity, this time associated to the Halpern iterations: $\ds \lim_{n\to\infty} \|x_n-Tx_n\|=0$, where $(x_n)_{n\geq 1}$ is defined by (\ref{Halpern-iterate-x}). By inspecting the proof of Theorem \ref{wittmann-thm} (and its generalizations),  it is easy to see that the first step is to obtain asymptotic regularity, and that this can be done in  a much more general setting.  

Thus, the following theorem, essentially contained in \cite{Wittmann,Xu-02-JLMS,Xu-viscosity}, can be proved.

\begin{theorem}\label{Halpern-ass-reg}
Let $(X,\|\cdot\|)$ be a normed space, $C\se X$  a nonempty convex subset and  $T:C\to C$ be nonexpansive. 

Assume that $(\lambda_n)_{n\geq 1}$ is a sequence in $[0,1]$  such that $\ds\lim_{n\to\infty} \lambda_n =0$, $\ds\sum_{n=1}^\infty \lambda_n$ is divergent and $\ds\sum_{n=1}^\infty|\lambda_{n+1}-\lambda_n|$ is convergent.

Let $x\in C$ be such that $(x_n)$ is bounded. 

Then 
\[\lim_{n\to\infty} \|x_n-Tx_n\|=0.\]
\end{theorem}

This theorem is our point of departure. By a logical analysis of its proof, we shall obtain a quantitative version (Theorem \ref{main-thm}), providing for the first time effective rates of asymptotic regularity for the Halpern iterates, that is rates of convergence of $\big(\|x_n-Tx_n\|\big)$ towards $0$.

\section{Main results}\label{main-results}

Before stating our main theorem, let us recall some terminology.

Let $(a_n)_{n\geq 1}$ be a sequence of real numbers. If the series $\ds \sum_{n=1}^\infty a_n$ is divergent, then a function $\gamma:\N^*\to\N^*$ is called a {\em rate of divergence} of $\ds \sum_{n=1}^\infty a_n$ if 
\beq
\forall n\in\N^*\left(\sum_{i=1}^{\gamma(n)}a_i \geq n\right).\label{def-rate-div}
\eeq
If $(a_n)_{n\geq 1}$ is convergent, then a function $\gamma:(0,\infty)\to\N^*$ is called a {\em Cauchy modulus} of $(a_n)$ if 
\beq
\forall \eps>0\,\forall n\in\N^*\left(a_{\gamma(\eps)+n}-a_{\gamma(\eps)} < \eps\right)\label{def-mod-Cauchy}.
\eeq
If $\ds\lim _{n\to\infty}a_n=a$, then a function $\gamma:(0,\infty)\to\N^*$ is called a {\em rate of convergence} of $(a_n)$ if 
\beq
\forall \eps>0\,\forall n\geq \gamma(\eps)\left(|a_n-a| < \eps\right)\label{def-rate-conv}.
\eeq

The following quantitative version of Theorem \ref{Halpern-ass-reg} is the main result of our paper.

\begin{theorem}\label{main-thm}
Let $(X,\|\cdot\|)$ be a normed space, $C\se X$  a nonempty convex subset and  $T:C\to C$ be nonexpansive. \\
Assume that $(\lambda_n)_{n\geq 1}$ is a sequence in $[0,1]$  such that $\ds\lim_{n\to\infty} \lambda_n =0$, $\ds\sum_{n=1}^\infty \lambda_n$ is divergent and $\ds\sum_{n=1}^\infty|\lambda_{n+1}-\lambda_n|$ is convergent.
Moreover, let $\alpha:(0,\infty)\to\N^*$ be a rate of convergence of $(\lambda_n)$, $\beta:(0,\infty)\to\N^*$ be a Cauchy modulus of $s_n:=\ds\sum_{i=1}^n|\lambda_{i+1}-\lambda_i|$ and $\theta:\N^*\to\N^*$ be a rate of divergence of $\ds\sum_{n=1}^\infty \lambda_n$.\\
Let $x\in C$ be such that $(x_n)$ is bounded. 

Then  $\ds\lim_{n\to\infty} \|x_n-Tx_n\|=0$ and moreover 
\bua
 \forall \eps\in(0,2) \forall n\ge \Phi(\alpha,\beta,\theta,M,\eps)\ \big( \|x_n-Tx_n\|< \eps\big), \label{conclusion-main-thm}
\eua
where 
\bua
\begin{array}{l}
\ds\Phi(\alpha,\beta,\theta,M,\eps)=\max\left\{\theta\left(\beta\left(\frac\eps {8M}\right)+1+\left\lceil\ln\left(\frac{8M}\eps\right)\right\rceil\right),\,\,\alpha\left(\frac\eps {4M}\right)\right\},\\[0.2cm]\label{def-Phi-main-thm}
M\in\N^* \text{~is such that~} M\geq \|x_n\|+\|x\|+\|Tx\| \text{~~for all~~} n\geq 1.
\end{array}
\eua
\end{theorem}

\noindent We shall give the proof of the above theorem in the last section of our paper. We derive now some further consequences. 

\begin{corollary}\label{C-bounded}
Let $(X,\|\cdot\|)$ be a normed space, $C\se X$  a nonempty convex bounded subset with finite diameter $d_C$ and  $T:C\to C$ be nonexpansive. \\
Assume that $(\lambda_n)_{n\geq 1}$ satisfies the hypotheses of Theorem \ref{main-thm}.

Then  $\ds\lim_{n\to\infty} \|x_n-Tx_n\|=0$ for all $x\in C$ and moreover 
\bua
 \forall \eps\in(0,2) \forall n\ge \Phi(\alpha,\beta,\theta,d_C,\eps)\ \big( \|x_n-Tx_n\|< \eps\big), 
\eua
where 
\bua
\begin{array}{l}
\ds\Phi(\alpha,\beta,\theta,d_C,\eps)=\max\left\{\theta\left(\beta\left(\frac\eps {8M}\right)+1+\left\lceil\ln\left(\frac{8M}\eps\right)\right\rceil\right),\,\,\alpha\left(\frac\eps {4M}\right)\right\},\\[0.2cm]\label{def-Phi-bounded}
M\in\N^* \text{~is such that~} M\geq 3d_C.
\end{array}
\eua
\end{corollary}
\begin{proof}
Since $C$ is bounded, it has a finite diameter $d_C:=\sup\{\|x\|\mid x\in C\}$.  Moreover, for all $x\in C$, $(x_n)$ is bounded and
$\|x_n\|+\|x\|+\|Tx\|\leq 3d_C$. Apply  now Theorem \ref{main-thm}. $\hfill\qed$\end{proof}

\noindent Thus, for bounded $C$, we get asymptotic regularity for general $(\lambda_n)$ and an explicit  rate of asymptotic regularity $\Phi(\alpha,\beta,\theta,d_C,\eps)$  which depends only on the error $\varepsilon$, on the diameter $d_C$ of $C$, and on $(\lambda_n)$  via $\alpha,\beta,\theta$, but not  on the nonexpansive mapping $T$, the starting point $x\in C$ of the Halpern iteration or other data related with $C$ and $X$.

\begin{corollary}\label{lambda-decreasing}
Let $(X,\|\cdot\|)$ be a normed space, $C\se X$ be a nonempty convex subset and  $T:C\to C$ nonexpansive. \\
Assume that $(\lambda_n)_{n\geq 1}$ is a decreasing sequence in $[0,1]$  such that $\ds\lim_{n\to\infty} \lambda_n =0$, $\ds\sum_{n=1}^\infty \lambda_n$ is divergent and let $\alpha:(0,\infty)\to\N^*$ be a rate of convergence of $(\lambda_n)$ and $\theta:\N^*\to\N^*$ be a rate of divergence of $\ds\sum_{n=1}^\infty \lambda_n$.\\
Let $x\in C$ be such that $(x_n)$ is bounded.

Then  $\lim \|x_n-Tx_n\|=0$ and moreover 
\bua
 \forall \eps\in(0,2)\forall n\ge \Psi(\alpha,\theta,M,\eps)\ \big( \|x_n-Tx_n\|< \eps\big), 
\eua
where 
\bua
\begin{array}{l}
\ds\Psi(\alpha,\theta,M,\eps)=\max\left\{\theta\left(\alpha\left(\frac\eps {8M}\right)+1+\left\lceil\ln\left(\frac{8M}\eps\right)\right\rceil\right),\,\,\alpha\left(\frac\eps {4M}\right)\right\}, \\[0.2cm]
M\in\N^* \text{~is such that~} M\geq \|x_n\|+\|x\|+\|Tx\| \text{~~for all~~} n\geq 1.
\end{array}\label{def-Phi-lambda-decreasing}
\eua
\end{corollary}
\begin{proof}
Remark that $(\lambda_n)$ decreasing implies that
\bua
s_n:=\sum_{i=1}^n|\lambda_{i+1}-\lambda_i|=\sum_{i=1}^n(\lambda_i-\lambda_{i+1})=\lambda_1-\lambda_{n+1}.
\eua
Since $\ds\lim_{n\to\infty} \lambda_n=0$, it follows that $\ds\sum_{n=1}^\infty|\lambda_{n+1}-\lambda_n|=\lambda_1$, that is it is convergent. Moreover, for all $\eps>0, n\in\N^*$,
\bua
s_{\alpha(\eps)+n}-s_{\alpha(\eps)}&= &(\lambda_1-\lambda_{\alpha(\eps)+n+1})-(\lambda_1-\lambda_{\alpha(\eps)+1})=\lambda_{\alpha(\eps)+1}-\lambda_{\alpha(\eps)+n+1}\\
&&\leq \lambda_{\alpha(\eps)+1}\leq \lambda_{\alpha(\eps)}< \eps,
\eua
since $\alpha$ is a rate of convergence of $(\lambda_n)$. Thus, $\alpha$ is a Cauchy modulus of $(s_n)$, so we can apply now Theorem \ref{main-thm} with $\beta:=\alpha$. $\hfill\qed$
\end{proof}

The rate of asymptotic regularity can be further simplified for $\lambda_n=1/n$.

\begin{corollary}
Let $(X,\|\cdot\|)$ be a normed space, $C\se X$  a nonempty convex bounded subset with finite diameter $d_C$ and  $T:C\to C$ be nonexpansive. \\[0.1cm]
Assume that $\ds\lambda_n=\frac 1n$ for all $n\geq 1$.

Then  $\ds\lim_{n\to\infty} \|x_n-Tx_n\|=0$ for all $x\in C$ and moreover 
\bua
 \forall \eps\in(0,2)\forall n\ge \Phi(d_C,\eps)\ \big( \|x_n-Tx_n\|< \eps\big), 
\eua
where 
\bua
\begin{array}{l}
\ds\Phi(d_C,\eps)=\exp\left(\ln 4\cdot\left(\frac{16M}\eps+3\right)\right),\\[0.2cm]
M\in\N^* \text{~is such that~} M\geq 3d_C.
\end{array}
\eua
\end{corollary}
\begin{proof}
Obviously, $\ds\lim_{n\to\infty}\frac1n=0$ with a rate of convergence 
\[\alpha:(0,\infty)\to\N^*, \quad \alpha(\eps)= \left\lceil\frac{1}\eps\right\rceil+1.\]
Moreover, $\ds\sum_{n=1}^\infty\frac 1n$ is divergent with a rate of divergence given by  
\[\theta:\N^*\to\N^*, \quad \theta(n)=4^{n}.\]
Since, furthermore, $\ds\left(\frac1n\right)$ is decreasing, we can apply Corollaries \ref{lambda-decreasing} and \ref{C-bounded} to get that $\lim \|x_n-Tx_n\|=0$ for all $x\in C$ and moreover 
\[\forall \eps\in(0,2)\forall n\ge \Psi(\alpha,\theta,M,\eps)\ \big( \|x_n-Tx_n\|< \eps\big),\]
where 
\bua
\Psi(\alpha,\theta,M,\eps)&=& \max\left\{\theta\left(\alpha\left(\frac\eps {8M}\right)+1+\left\lceil\ln\left(\frac{8M}\eps\right)\right\rceil\right),\,\,\alpha\left(\frac\eps {4M}\right)\right\}\\
&=& \theta\left(\alpha\left(\frac\eps {8M}\right)+1+\left\lceil\ln\left(\frac{8M}\eps\right)\right\rceil\right)
\eua
and $M\in\N^*$ is such that $M\geq 3d_C$.
Using that $\lceil a\rceil < a+1$ and $1+\ln a\leq a$ for all $a>0$, we get that
\bua
\alpha\left(\frac\eps {8M}\right)+1+\left\lceil\ln\left(\frac{8M}\eps\right)\right\rceil &<& \alpha\left(\frac\eps {8M}\right)+2+\ln\left(\frac{8M}\eps\right)\leq\alpha\left(\frac\eps {8M}\right)+1+\frac{8M}\eps\\
&=& \left\lceil\frac{8M}\eps\right\rceil+2+\frac{8M}\eps
< \frac{16M}\eps+3,
\eua
we get that
\bua
\Psi(\alpha,\theta,M,\eps)&<& 4^{\frac{16M}\eps+3}=\exp\left(\ln 4\cdot\left(\frac{16M}\eps+3\right)\right)=\Phi(d_c,\eps).
\eua
The conclusion follows now immediately. $\hfill\Box$
\end{proof}

\noindent Hence, we get an exponential (in $1/\eps$) rate of asymptotic regularity in the case $\lambda_n=1/n$.

\section{Some technical lemmas}

The following lemma collects some useful properties of Halpern iterations and it is essentially contained in \cite{Xu-02-JLMS,Xu-viscosity}. In order to make the paper self-contained, we still give the proof.

\begin{lemma}\label{lemma-Halpern}
Let $(X,\|\cdot\|)$ be a normed space, $C\se X$ be a nonempty convex subset,  $T:C\to C$ nonexpansive and $(\lambda_n)_{n\geq 1}$ be a sequence in $[0,1]$. 
Assume that $(x_n)_{n\geq 1}$ is the Halpern iteration starting with $x\in C$. Then
\be
\item\label{lemma-Halpern-1}  For all $n\geq 1$,
\bua
\|Tx_n\|&\leq &\|x_n\|+\|x\|+\|Tx\|,\\
\|Tx_n-x_n\|&\leq & \|x_{n+1}-x_n\|+\lambda_{n+1}\|x-Tx_n\|,\\
\|x_{n+1}-x_n\| &\leq & (1-\lambda_{n+1})\|x_n-x_{n-1}\|+|\lambda_{n+1}-\lambda_n|\cdot\|x-Tx_{n-1}\|.
\eua
\item \label{lemma-Halpern-2} If $(x_n)$ is bounded, then  $(Tx_n)$ is also bounded. Moreover, if $M\geq  \|x_n\|, \|Tx_n\|$ for all $n\geq 1$, then
\bua
\|Tx_n-x_n\|&\leq &\|x_{n+1}-x_n\| +2M\lambda_{n+1},\label{ineq-Txn-xn}\\
\|x_{n+1}-x_n\| &\leq & (1-\lambda_{n+1})\|x_n-x_{n-1}\|+2M|\lambda_{n+1}-\lambda_n| \label{ineq-xn+1-xn}
\eua
for all $n\geq 1$.
\ee 
\end{lemma}
\begin{proof}
\be
\item
\bua
\|Tx_n\|&\leq& \|Tx_n-Tx\|+\|Tx\|\leq \|x_n-x\|+\|Tx\|\leq \|x_n\|+\|x\|+\|Tx\|\\
\|Tx_n-x_n\|&=& \|(\lambda_{n+1}x+(1-\lambda_{n+1})Tx_n-\lambda_{n+1}(x-Tx_n))-x_n\|\\
&=&\|x_{n+1}-x_n-\lambda_{n+1}(x-Tx_n)\|\leq \|x_{n+1}-x_n\|+\lambda_{n+1}\|x-Tx_n\|
\eua
\bua
\|x_{n+1}-x_n\| & = &  \|\lambda_{n+1}x+(1-\lambda_{n+1})Tx_n - \lambda_{n}x-(1-\lambda_{n})Tx_{n-1}\|\\
 & = &  \|(\lambda_{n+1}-\lambda_n)x+(1-\lambda_{n+1})(Tx_n-Tx_{n-1})+
(\lambda_n-\lambda_{n+1})Tx_{n-1}\|\\
 & = & \|(\lambda_{n+1}-\lambda_n)(x-Tx_{n-1})+(1-\lambda_{n+1})(Tx_n-Tx_{n-1})\|\\
&\leq & |\lambda_{n+1}-\lambda_n|\cdot \|x-Tx_{n-1}\|+(1-\lambda_{n+1})\|x_n-x_{n-1}\|,\\
&& \text{since~}  T \text{~is nonexpansive}.
\eua
\item is an immediate consequence of \ref{lemma-Halpern-1}.
\ee
\end{proof}

\begin{lemma}\label{lemma-an-1}
Let $(\lambda_n)_{n\ge 1}$ be a sequence in $[0,1]$ and $(a_n)_{n\geq 1},(b_n)_{n\geq 1}$ be sequences in $\R_+$  such that $\ds \sum_{n=1}^\infty b_n$ is convergent and 
\[
a_{n+1}\leq (1-\lambda_{n+1}) a_n + b_n \quad \text{for all~} n\in\N^*.
\]
Then
\be
\item for all $m,n\in\N^*$,
\beq
a_{n+m}\leq \left[\prod_{j=n}^{n+m-1}(1-\lambda_{j+1})\right]a_n+\sum_{j=n}^{n+m-1}b_j\label{ineq-prod-a-n}
\eeq
\item $(a_n)$ is bounded.
\ee
\end{lemma}
\begin{proof}
\be
\item By an easy induction on $m$.
\item Applying (\ref{ineq-prod-a-n}) with $n:=1$, we get that for all $m\geq 1$,
\[0\leq a_{m+1}\leq \left[\prod_{j=1}^{m}(1-\lambda_{j+1})\right]a_1+\sum_{j=1}^{m}b_j\leq a_1+\sum_{j=1}^m b_j\leq a_1+\sum_{j=1}^\infty b_j< \infty,\]
since $\ds\sum_{j=1}^\infty b_j <\infty$. Thus, $(a_n)$ is bounded.
\qed
\ee
\end{proof}

The following lemma is a quantitative version of \cite[Lemma 2]{Liu}. 

\begin{lemma}\label{quant-liu}$\,$\\
Let $(\lambda_n)_{n\ge 1}$ be a sequence in $[0,1]$ and $(a_n)_{n\geq 1},(b_n)_{n\geq 1}$ be sequences in $\R_+$  such that for all $\ds n\in\N^*$,
\beq
a_{n+1}\leq (1-\lambda_{n+1}) a_n + b_n.
\eeq
Assume moreover that $\ds \sum_{n=1}^\infty \lambda_n$ is divergent,$\,\ds \sum_{n=1}^\infty b_n$ is convergent and let $\delta:\N^*\to\N^*$ be a rate of divergence of $\ds \sum_{n=1}^\infty \lambda_n$, $\gamma:(0,\infty)\to\N^*$ be a Cauchy modulus of $(s_m)_{m\geq 1}$, where $s_m:=\ds\sum_{i=1}^mb_i$. 

Then $\ds\lim_{n\to\infty} a_n =0$ and moreover 
\beq
 \forall \eps\in(0,2) \forall n\ge
h(\gamma,\delta,D,\eps)\ \big(a_n < \eps\big), \label{conclusion-quant-liu}
\eeq
where 
\[
\begin{array}{l}
\ds h(\gamma,\delta,D,\eps)=\delta\left(\gamma\left(\frac\eps 2\right)+1+\left\lceil\ln\left(\frac{2D}\eps\right)\right\rceil\right),\\[0.2cm]
D\in\N^* \text{~ is an upper bound on~} (a_n).
\end{array}
\]
\end{lemma}
\begin{proof}
By Lemma \ref{lemma-an-1}, $(a_n)$ is bounded, so there exists $D\in \N^*$ such that $a_n\leq D$ for all $n\in\N^*$.
Let $\eps\in(0,2)$ and define 
\beq
N:=\gamma\left(\frac\eps 2\right)+1. \label{def-N}
\eeq
Applying (\ref{ineq-prod-a-n}) with $n:=N$, it follows that for all $m\in\N^*$
\bua
a_{N+m}&\leq & \left[\prod_{j=N}^{N+m-1}(1-\lambda_{j+1})\right]a_N+\sum_{j=N}^{N+m-1}b_j\\
&\leq & \exp\left(-\sum_{j=N}^{N+m-1}\lambda_{j+1}\right)a_N+\sum_{j=N}^{N+m-1}b_j,\\
&&\,\, \text{since~}1-x\leq \exp(-x) \text{~for all~}x\in[0,\infty)\\
&=&\exp\left(-\sum_{j=N}^{N+m-1}\lambda_{j+1}\right)a_N+\left(s_{\gamma\left(\frac\eps 2\right)+m}-s_{\gamma\left(\frac\eps 2\right)}\right)\\
&< & D\exp\left(-\sum_{j=N}^{N+m-1}\lambda_{j+1}\right)a_N+\frac\eps 2,\\
&& \text{since~} \gamma \text{~is a Cauchy modulus of~} (s_m).
\eua
For simplicity, let us denote $\ds d_m:=D\exp\left(-\sum_{j=N}^{N+m-1}\lambda_{j+1}\right)$.
We have got then that for all $m\in\N^*$,
\beq
a_{N+m} < d_m + \frac\eps 2. \label{ineq-aN+m}
\eeq
Let us note that 
\bua
d_m \leq \frac\eps 2 & \Lra & \exp\left(-\sum_{j=N}^{N+m-1}\lambda_{j+1}\right)\leq \frac\eps{2D} \Lra  -\sum_{j=N}^{N+m-1}\lambda_{j+1}\leq \ln\left(\frac\eps{2D}\right)\\
&\Lra & \sum_{j=N}^{N+m-1}\lambda_{j+1}\geq -\ln\left(\frac\eps{2D}\right)=\ln\left(\frac{2D}\eps\right)\Lra \sum_{i=N+1}^{N+m}\lambda_{i}\geq \ln\left(\frac{2D}\eps\right)\\
&\Lra &\sum_{i=1}^{N+m}\lambda_{i}\geq \sum_{i=1}^{N}\lambda_i+\ln\left(\frac{2D}\eps\right).
\eua
Let
\beq
M:= \delta\left(N+\left\lceil\ln\left(\frac{2D}\eps\right)\right\rceil\right)-N.
\eeq
Since $\delta$ is a rate of divergence of $\ds\sum_{n=1}^\infty \lambda_n$ and $\lambda_n\leq 1$, it is obvious that $\delta(n)\geq n$ for all $n\in\N^*$. Using also the fact that $\ds\frac{2D}{\eps}> D >1$, so $\ds\ln\left(\frac{2D}\eps\right)>0$, it is easy to see that $M\in \N^*$. Moreover, for $m\ge M$, we get that
\bua
 \sum_{i=1}^{N+m}\lambda_{i}\geq \sum_{i=1}^{N+M}\lambda_{i}\geq N+\left\lceil\ln\left(\frac{2D}\eps\right)\right\rceil\geq \sum_{i=1}^{N}\lambda_i+\ln\left(\frac{2D}\eps\right).
\eua
Hence, $\ds d_m \leq \frac\eps 2$ for all $m\geq M$. Combining this with (\ref{ineq-aN+m}), we get that for all $m\geq M$, $a_{N+m} <\eps$, that is 
\beq
a_{N+M+n}<\eps.
\eeq
for all $n\in\N$.
Define
\beq
h(\gamma,\delta,D,\eps):=N+M= \delta\left(N+\left\lceil\ln\left(\frac{2D}\eps\right)\right\rceil\right)
\eeq
Then (\ref{conclusion-quant-liu}) follows. Thus, $\lim a_n=0$ and $h(\gamma,\delta,D,\eps)$ is a rate of convergence of $(a_n)$ towards $0$.
\end{proof}

\section{Proof of Theorem \ref{main-thm}}

By Lemma \ref{lemma-Halpern}, we get that $M\geq \|x_n\|, \|Tx_n\|$ for all $n\geq 1$ and 
\beq
\|x_{n+1}-x_n \|\leq  (1-\lambda_{n+1})\|x_n-x_{n-1}\|+2M\cdot|\lambda_{n+1}-\lambda_n|.\label{ineq-1}
\eeq
Let us consider the sequences
\[a_n:=\|x_n-x_{n-1}\|,  \quad b_n:=2M|\lambda_{n+1}-\lambda_n|\]
and let  $D:=2M$. Then $D$ is a bound on $(a_n)$ and, by (\ref{ineq-1}), for all $n\geq 1$,
\[a_{n+1}\leq (1-\lambda_{n+1})a_n+b_n.\]
Moreover, $\ds\sum_{n=1}^\infty \lambda_n$ is divergent with rate of divergence $\theta$ and if we define
\[\gamma:(0,\infty)\to \N^*, \quad \gamma(\eps):=\beta\left(\frac\eps{2M}\right),\]
we get that for all $n\in \N^*$,
\bua
\sum_{i=1}^{\gamma(\eps)+n}b_i-\sum_{i=1}^{\gamma(\eps)}b_i&=& 2M\left(\sum_{i=1}^{\gamma(\eps)+n}|\lambda_{i+1}-\lambda_i|-\sum_{i=1}^{\gamma(\eps)}|\lambda_{i+1}-\lambda_i|\right)\\
&=& 2M\left(s_{\beta\left(\frac\eps{2M}\right)+n}-s_{\beta\left(\frac\eps{2M}\right)}\right)\\
&<& 2M\cdot\frac\eps{2M}=\eps,
\eua
so $\ds \sum_{n=1}^\infty b_n$ is convergent and $\gamma$ is a Cauchy modulus of $\left(\ds\sum_{i=1}^nb_i\right)$.

Thus, the hypothesis of Lemma \ref{quant-liu} are satisfied, so we can apply it to get that for all $\eps\in(0,2)$ and for all $n\geq h_1(\beta,\theta,M,\eps)$
\beq
\|x_n-x_{n-1}\|<\frac\eps 2,\label{final-1}
\eeq
where 
\bua
h_1(\beta,\theta,M,\eps)&:=& \theta\left(\beta\left(\frac\eps {8M}\right)+1+\left\lceil\ln\left(\frac{8M}\eps\right)\right\rceil\right).
\eua
By Lemma \ref{lemma-Halpern}.\ref{lemma-Halpern-2}, for all $n\geq 2$,
\beq
\|x_{n-1}-Tx_{n-1}\| \leq  \|x_n-x_{n-1}\|+2M\lambda_n. \label{final-0}
\eeq
Let $\ds h_2(\alpha,M,\eps):=\alpha\left(\frac\eps {4M}\right)$. Then, using the fact that $\alpha$ is a rate of convergence of $(\lambda_n)$ towards $0$, we get that for all $n\geq h_2(\alpha,M,\eps)$
\beq
2M\lambda_{n}< 2M\frac\eps {4M}=\frac\eps 2.\label{final-2}
\eeq
Combining (\ref{final-1}), (\ref{final-0}) and (\ref{final-2}), it follows that 
\[\|x_{n-1}-Tx_{n-1}\|<\eps\]
for all $n\geq \max\{h_1(\beta,\theta,M,\eps),h_2(\alpha,M,\eps)\}$, so the conclusion of the theorem follows with $\Phi$ defined by (\ref{def-Phi-main-thm}).
$\hfill\qed$


\begin{thebibliography}{5}

\bibitem{Borwein/Reich/Shafrir(92)}
Borwein, J., Reich, S., Shafrir, I., rasnoselski-Mann iterations in normed spaces; Canad. Math. Bull 35 (1992),  21--28.

\bibitem{Browder/Petryshyn(66)} 
Browder, F.E., Petryshyn, W.V., The solution by iteration of nonlinear functional equations in Banach spaces,  Bull. Amer. Math. Soc. 72 (1966), 571--575. 


\bibitem{Edelstein(70)}
Edelstein, M., A remark on a theorem of M. A. Krasnoselskii,  Amer. Math. Monthly 14 (1970), 65--73.

\bibitem{Gerhardy/Kohlenbach} 
Gerhardy, P., Kohlenbach, U., General logical metatheorems for functional 
analysis, Trans. Amer. Math. Soc. (2007),  doi:  10.1090/S0002-9947-07-04429-7 

\bibitem{Groetsch(72)} 
Groetsch, C.W., A note on segmenting Mann iterates, J. Math. Anal. and Appl. 40 (1972), 369--372.

\bibitem{Halpern(67)}
Halpern, B., Fixed points of nonexpanding maps, Bull. Amer. Math. Soc. 73 (1967), 957-961.

\bibitem{Kirk-handbook}
Kirk, W.A., Sims, B. (eds.), Handbook of Metric Fixed Point Theory, Kluwer, 2001.

\bibitem{Kohlenbach(01)} 
Kohlenbach, U., A quantitative version of a theorem due to Borwein-Reich-Shafrir, Numer. Funct. Anal. and Optimiz. 22 (2001),  641--656.

\bibitem{Kohlenbach(03)}
Kohlenbach, U., Uniform asymptotic regularity for Mann iterates,  J. Math. Anal. and Appl. 279 (2003), 531--544.

\bibitem{Kohlenbach-metapaper} Kohlenbach, U., Some logical metatheorems with applications in functional analysis, Trans. Amer. Math. Soc. 357 
(2005), 89-128.

\bibitem{K-book} Kohlenbach, U., Applied Proof Theory: Proof Interpretations and their Use in Mathematics, to appear in Springer Monographs in Mathematics.
`
\bibitem{Kohlenbach/Leustean(03)}
Kohlenbach, U., Leu\c{s}tean, L., Mann iterates of directionally nonexpansive mappings in hyperbolic spaces, Abstract and Applied Analysis 2003 (2003),  449--477.

\bibitem{Krasnoselski(55)} 
Krasnoselski, M.A., Two remarks on the method of successive approximation,  Usp. Math. Nauk (N.S.) 10 (1955), 123--127 (Russian).

\bibitem{LL-WOLLIC-06}
Leu\c stean, L., Proof mining in $\R$-trees and hyperbolic spaces,  Electronic Notes in Theoretical Computer Science 165 (2006), 95--106.

\bibitem{L-07-JMAA}
Leu\c stean, L., A quadratic rate of asymptotic regularity for CAT(0)-spaces,  J. Math. Anal. Appl. 325 (2007), 386-399.

\bibitem{Lions}
Lions, P.L., Approximation de points fixes de contractions, C. R. Acad. Sci. Paris S\' er. A 284 (1977), 1357-1359.

\bibitem{Liu}
Liu, L.-S., Ishikawa and Mann iterative process with errors for nonlinear strongly accretive mappings in Banach spaces, J. Math. Anal. Appl. 194 (1995), 114-125.

\bibitem{Mann(53)}
Mann, W.R., Mean value methods in iteration, Proc. Amer. Math. Soc. 4 (1953),  506--510.

\bibitem{Schaefer(57)}
Schaefer, H., Uber die Method sukzessive Approximationen, J. Deutsch Math. Verein 5 (1957), 131--140.

\bibitem{Wittmann-Hopf}
Wittmann, R., Hopf ergodic theorem for nonlinear operators, Math. Ann. 289 (1991), 239-253.

\bibitem{Wittmann}
Wittmann, R., Approximation of fixed points of nonexpansive mappings, Arch. Math. 58 (1992), 486-491.

\bibitem{Xu-02-JLMS}
Xu, H.-K., Iterative algorithms for nonlinear operators, J. London Math. Soc. 66 (2002), 240--256.

\bibitem{Xu-viscosity}
Xu,  H.-K., Viscosity approximation methods for nonexpansive mappings, J. Math. Anal. Appl. 298 (2004), 279--291.

\bibitem{Xu-07-JMAA-online}
Xu, H.-K., A strong convergence theorem for nonexpansive mappings, J. Math. Anal. Appl. (2007), doi:10.1016/j.jmaa.2007.03.78.

\end{thebibliography}
\end{document}